\documentclass[12pt,a4paper,leqno,oneside]{article}
\usepackage{amsmath,amssymb,amsfonts,amsthm}
 \newcommand{\dtv}[2]{d_{TV}\left(#1,#2\right)}
 \newcommand{\whp}{with high probability}
 \newcommand{\V}{\mathcal{V}}
\newcommand{\W}{\mathcal{W}}
 \newcommand{\Po}[1]{\textrm{Po}\left(#1\right)}
 \newcommand{\Bin}[2]{\textrm{Bin}\left(#1,#2\right)}
 \newcommand{\E}{\mathbb{E}}

\newcommand{\coupling}{\preceq}
\newcommand{\coup}{\preceq_{1-o(1)}}
\newcommand{\Pra}[1]{\Pr\left\{#1\right\}}
\newcommand{\Gn}[1]{G\left(n,#1\right)}
\newcommand{\Gnm}[1]{\mathcal{G}\left(n,m,#1\right)}
\newcommand{\G}[1]{\mathcal{G}_*\left(n,#1\right)}

\newtheorem{tw}{Theorem}

\newtheorem{lem}{Lemma}

\newtheorem{fact}{Fact}
\theoremstyle{definition}
\newtheorem{rem}{Remark}

\newtheorem{conj}{Conjecture}
\theoremstyle{plain}

\begin{document}

\author{Katarzyna Rybarczyk$^*$}
\title{Sharp threshold functions for the random intersection graph via coupling method.}
\date{}
\maketitle
\begin{center}
$^*$Faculty of Mathematics and Computer
Science, Adam Mickiewicz University,\\ 60--769 Pozna\'n, Poland
\end{center}

\begin{abstract}
We will present a new method, which enables us to find threshold functions for many properties in random intersection graphs. This method will be used to establish sharp threshold functions in random intersection graphs for $k$--connectivity, perfect matching containment and Hamilton cycle containment. 
\end{abstract}

{\bfseries keywords:} random intersection graph, threshold functions, connectivity, Hamilton cycle, perfect matching, coupling

\section{Introduction}\label{Introduction}

In a random intersection graph edges represent relations between feature sets randomly attributed to vertices. More precisely, in a random intersection graph each vertex $v$ from the vertex set $\V$ ($|\V|=n$) is assigned independently a subset of features $W_v\subseteq\W$ from the auxiliary set of features $\W$ ($|\W|=m$) according to a given probability distribution. Two vertices  $v_1$, $v_2$ are adjacent in a random intersection graph if and only if
$W_{v_1}$ and  $W_{v_2}$ intersect. Such a general model of the random intersection graph was introduced in \cite{RIGGodehardt1}. In the article we will concentrate on analysing the properties of the most widely studied random intersection graph $\Gnm{p}$, in which for any vertex $v\in\V$ each feature $w\in\W$ is added  to $W_v$ independently with probability $p$ 
(i.e. $\Pr\{w\in W_v\}=p$). The $\Gnm{p}$ model was introduced in \cite{GpSubgraph,SingerPhD}. We will also make a standard assumption that the number of vertices and the number of features are in the relation $m=n^{\alpha}$, where $\alpha$ is a positive constant. We should mention here that, to some extent, the results obtained in the article may be generalised to other random intersection graph models due to the equivalence theorems proved in Section~4 in \cite{WSNphase2}.

The difference between $\Gn{\hat{p}}$, in which each edge appears independently with probability $\hat{p}$,  and $\Gnm{p}$ is unquestionable. This is caused by the dependencies of edge appearance in the latter one. However
known results suggest that there is some relation between threshold functions of $\Gnm{p}$ and  $\Gn{\hat{p}}$, while $\hat{p}$ is approximately $mp^2$ and $\alpha$ is large enough (i.e. $m$ is large comparing to $n$). As the example we may state the equivalence theorem from~\cite{GpEquivalence}, according to which $\Gnm{p}$ and $\Gn{\hat{p}}$ have asymptotically the same properties as $m=n^{\alpha}$ for $\alpha>6$ and  $\hat{p}$ is specially defined function depending on $p$.  In~\cite{GpEquivalence2} it is shown that the equivalence theorem from \cite{GpEquivalence} is true also for $\alpha\ge 3$ in  the case of monotone properties. However any equivalence theorem in such general form is false for $\alpha$ smaller than $3$. This is caused by the excess of the number of cliques in $\Gnm{p}$ comparing to the number of cliques in $\Gn{\hat{p}}$ (see \cite{GpSubgraph,GpPoissonCliques}). However the values of the threshold functions for connectivity \cite{SingerPhD} and phase transition \cite{GpPhaseTransition} suggest that some comparison is still possible for $\alpha>1$. 

In this article we will introduce a new technique basing on coupling, which shows the relation between $\Gn{\hat{p}}$ and $\Gnm{p}$ models for all values of $\alpha$. We will use it to give an alternative short proof of the connectivity theorem shown in~\cite{SingerPhD}. From the proof it will clearly follow, why the threshold functions for $\alpha>1$ coincide. We will also use the technique to prove new results concerning sharp threshold functions for  Hamilton cycle containment, perfect matching containment and $k$-connectivity. All of these graph properties  follow the so called 'minimum degree phenomenon' in $\Gn{\hat{p}}$. This means that, with probability tending to $1$ as $n$ tends to infinity, the properties hold in $\Gn{\hat{p}}$ as soon as their necessary minimum degree condition is satisfied. In fact we will show that the 'minimum degree phenomenon' holds also in the case of $\Gnm{p}$ for $\alpha>1$. 

In the proof we will find the value of $\hat{p}$ and a coupling $(\Gn{\hat{p}},\Gnm{p})$ such that $\Gn{\hat{p}}\subseteq \Gnm{p}$ with probability tending to one as $n$ tends to infinity. Then we will use the coupling to bound the values of the threshold functions in $\Gnm{p}$ of the graph properties mentioned above  and we will prove that this values coincide with the values of the threshold functions of the minimum degree condition.
Our work is partially inspired by the result of Efthymiou and Spirakis~\cite{SpirakisHamiltonCycles}. However the method differs much from this used in~\cite{SpirakisHamiltonCycles} and therefore it enables us to obtain much sharper threshold functions in the case of the Hamilton cycle containment property then those from~\cite{SpirakisHamiltonCycles}. 
We should mention here that the method is strong enough to give some partial results on threshold function of other properties of $\Gnm{p}$. However we present here as an example these graph properties for which the threshold functions obtained by coupling method are tight.

The article is organised as follows. First, in Section~\ref{SectionGnp}, we present the well known results on threshold functions in $\Gn{\hat{p}}$. They will be useful later in the proof, however they are stated first as the comparison to the main results. In Section~\ref{SectionResult} we present the main theorems and outline their proof. Sections~\ref{SectionMinimumDegree} and \ref{SectionCoupling} are to give the details of the reasoning.

All limits in the paper are taken as $n\rightarrow \infty$.
Throughout the paper we will use the notation
$a_n=o(b_n)$
if $a_n/b_n\to 0$. Also by $\Bin{n}{p}$ and $\Po{\lambda}$ we will denote the binomial distribution with parameters $n$, $p$ and Poisson distribution with expected value $\lambda$, respectively. Moreover if a random variable $X$ is stochastically dominated by $Y$ we will write $X\prec Y$. We will also use the phrase '{\whp}' to say with probability tending to one as $n$ tends to infinity.
 
\section{Threshold functions in $\Gn{\hat{p}}$}\label{SectionGnp}

In this section we present classical results concerning threshold functions  of the properties, which follow the 'minimum degree phenomenon' in $\Gn{\hat{p}}$. They will be used in the proof of analogous theorems concerning $\Gnm{p}$.

\begin{tw}[Erd\H{o}s and R\'{e}nyi~\cite{Erdos1} see also Bollob\'{a}s and Thomason~\cite{GnpPrzeglad}]\ \\
Let
$$
\hat{p}=\frac{\ln n + \omega}{n}.
$$ 
\begin{itemize}
 \item[(i)] If $\omega\to -\infty$, then {\whp} $\Gn{\hat{p}}$ is disconnected.
\item[(ii)] If $\omega\to \infty$, then {\whp} $\Gn{\hat{p}}$ is connected.
\end{itemize}
\end{tw}

\begin{tw}[Erd\H{o}s and R\'{e}nyi~\cite{GnpPerfectMatching} see also Bollob\'{a}s and Thomason~\cite{GnpPrzeglad}]\ \\
Let $n$ be even and
$$
\hat{p}=\frac{\ln n + \omega}{n}.
$$ 
\begin{itemize}
 \item[(i)] If $\omega\to -\infty$, then {\whp} $\Gn{\hat{p}}$ does not contain a perfect matching.
\item[(ii)] If $\omega\to \infty$, then {\whp} $\Gn{\hat{p}}$ contains a perfect matching.
\end{itemize}
\end{tw}

\begin{tw}[Erd\H{o}s and R\'{e}nyi~\cite{GnpKConnectivity} see also Bollob\'{a}s and Thomason~\cite{GnpPrzeglad}]\ \\
Let $k\ge 1$ and
$$
\hat{p}_k=\frac{\ln n + (k-1)\ln \ln n + \omega}{n}.
$$ 
\begin{itemize}
 \item[(i)] If $\omega\to -\infty$, then {\whp} $\Gn{\hat{p}_k}$ is not $k$-connected.
\item[(ii)] If $\omega\to \infty$, then {\whp} $\Gn{\hat{p}_k}$ is $k$-connected.
\end{itemize}
\end{tw}

\begin{tw}[Koml\'{o}s and Szem\'{e}redi~\cite{GnpHamilton1} and Bollob\'{a}s~\cite{GnpHamilton2}]
Let
$$
\hat{p}=\frac{\ln n + \ln \ln n + \omega}{n}.
$$ 
\begin{itemize}
 \item[(i)] If $\omega\to -\infty$, then {\whp} $\Gn{\hat{p}}$ does not contain a Hamilton cycle.
 \item[(ii)] If $\omega\to \infty$, then {\whp} $\Gn{\hat{p}}$ contains a Hamilton cycle.
\end{itemize}
\end{tw}

\section{Result}\label{SectionResult}

We will show that, to some extent, $\Gnm{p}$ follows the 'minimum degree phenomenon'. This fact will be used to indicate threshold functions for $k$--connectivity, perfect matching containment and Hamilton cycle containment. 

The first of the presented results, Theorem~\ref{TheoremConnectivity}, was obtained in \cite{SingerPhD}. However we state it here, since our coupling method shortens the proof.

\begin{tw}\label{TheoremConnectivity}
Let $m=n^{\alpha}$ and
$$
p_1=
\begin{cases}
\frac{\ln n + \omega}{m},&\text{ for } \alpha\le 1;\\ 
\sqrt{\frac{\ln n + \omega}{nm}},&\text{ for } \alpha > 1.\\ 
\end{cases}
$$
\begin{itemize}
 \item[(i)] If $\omega\to -\infty$, then {\whp} $\Gnm{p_1}$ is disconnected.
\item[(ii)] If $\omega\to \infty$, then {\whp} $\Gnm{p_1}$ is connected.
\end{itemize}
\end{tw}

\noindent To the best of our knowledge, the following results (i.e. Theorems~\ref{TheoremMatching}, \ref{TheoremAlphaGe1} and~\ref{TheoremAlphaLe1})  were not known before.

\begin{tw}\label{TheoremMatching}
Let $n$ be even, $m=n^{\alpha}$ and
$$
p_1=
\begin{cases}
\frac{\ln n + \omega}{m},&\text{ for } \alpha\le 1;\\ 
\sqrt{\frac{\ln n + \omega}{nm}},&\text{ for } \alpha > 1.\\ 
\end{cases}
$$
\begin{itemize}
 \item[(i)] If $\omega\to -\infty$, then {\whp} $\Gnm{p_1}$ does not contain a perfect matching.
\item[(ii)] If $\omega\to \infty$, then {\whp} $\Gnm{p_1}$ contains a perfect matching.
\end{itemize}
\end{tw}

\begin{tw}\label{TheoremAlphaGe1}
Let $k\ge 1$ be a constant, $\alpha>1$, $m=n^{\alpha}$ and
$$
p_k=
\sqrt{\frac{\ln n + (k-1)\ln \ln n + \omega}{mn}}.
$$
\begin{itemize}
\item[(i)] If $\omega\to -\infty$, then {\whp} $\Gnm{p_k}$ is not $k$-connected.
\item[(ii)] If $\omega\to \infty$, then {\whp} $\Gnm{p_k}$ is $k$-connected.
\item[(i')] If $\omega\to -\infty$, then {\whp} $\Gnm{p_2}$ does not contain a Hamilton cycle.
\item[(ii')] If $\omega\to \infty$, then {\whp} $\Gnm{p_2}$ contains a Hamilton cycle.
\end{itemize}
\end{tw}

\begin{tw}\label{TheoremAlphaLe1}
Let $k\ge 1$ be a constant, $\alpha\le 1$, $m=n^{\alpha}$,
\begin{align*}
p_k&=
\frac{\ln n + (k-1)\ln \ln n + \omega}{m}.
\end{align*}
\begin{itemize}
\item[(i)] If $\omega\to -\infty$, then {\whp} $\Gnm{p_1}$ is not $k$-connected.
\item[(ii)] If $\omega\to \infty$, then {\whp} $\Gnm{p_k}$ is $k$-connected.
\item[(i')] If $\omega\to -\infty$, then {\whp} $\Gnm{p_1}$ does not contain a Hamilton cycle.
\item[(ii')] If $\omega\to \infty$, then {\whp} $\Gnm{p_2}$ contains a Hamilton cycle.
\end{itemize}
\end{tw}

\begin{rem}
Let $\mathcal{G}'(n,m,d)$ be a random intersection graph in which for all $v\in\V$ a feature set $D(v)$ is chosen uniformly at random from all $d$--element subsets of $\W$. This graph is sometimes called a uniform random intersection graph and is used to model wireless sensor networks with random predistribution of keys~(see for example \cite{WSNphase2,WSNWlosi}). By Lemma~4 in~\cite{WSNphase2} Theorems~\ref{TheoremConnectivity}, \ref{TheoremMatching} and~\ref{TheoremAlphaGe1} hold true, if we assume that $\alpha>1$ and replace $p_k$ by $d_k=mp_k$ and $\Gnm{p_k}$ by $\mathcal{G}'(n,m,d_k)$. Moreover by Lemma~3 in~\cite{WSNphase2} these theorems apply to even wider class of the random intersection graphs.
\end{rem}

\begin{proof}[Outline of the proof of Theorems~\ref{TheoremConnectivity}--\ref{TheoremAlphaLe1}]
Denote by $\deg(v)$ the degree of the vertex $v\in \V$ in $\Gnm{p}$ and by $\delta(\Gnm{p})=\min_{v\in \V}deg(v)$ the minimum degree of the graph.  The necessary condition for the $k$--connectivity is minimum degree at least $k$. In the case of perfect matching and Hamilton cycle containment the necessary condition is minimum degree at least $1$ and~$2$, respectively. Therefore the following two lemmas imply part (i) and (i') of the theorems. 
\begin{lem}\label{LemmaMinDegree}
Let $k\ge 1$ be a constant integer, $\alpha>1$ and  
$$
p_k=\sqrt{\frac{\ln n + (k-1)\ln \ln n + \omega}{nm}},
$$ 
\begin{itemize}
 \item[(i)] If $\omega\to -\infty$
then {\whp}
$\delta(\Gnm{p_k}) < k$ 
 \item[(ii)]
If $\omega\to \infty$
then {\whp}
$\delta(\Gnm{p_k}) \ge  k$ 
\end{itemize}
\end{lem}
\begin{lem}~\label{LemmaStopienMaleAlfa}
Let $\alpha\le 1$ and 
$$
p_1=\frac{\ln n + \omega}{m}.
$$ 
\begin{itemize}
 \item[(i)] If $\omega\to -\infty$
then {\whp}
$\delta(\Gnm{p_1})=0$. 
 \item[(ii)]
If $\omega\to \infty$
then {\whp}
$\delta(\Gnm{p_1})\ge 1$.
\end{itemize}
\end{lem}
\noindent The proof of the first lemma will be the subject of Section~\ref{SectionMinimumDegree} and the second lemma was shown in \cite{SingerPhD}. 

In the proof of the part (ii) and (ii') of the theorems we will use the fact that $k$--connectivity, Hamilton cycle containment and perfect matching containment are all increasing properties. Remind that for a family $\mathcal{G}$ of graphs with a vertex set $\V$, we  call $\mathcal{A}\subseteq \mathcal{G}$ an increasing property if $\mathcal{A}$ is closed under isomorphism and $G\in \mathcal{A}$ implies $G'\in \mathcal{A}$ for all $G'\in\mathcal{G}$ such that $E(G)\subseteq E(G')$. The following lemma will be shown in Section~\ref{SectionCoupling}.
\begin{lem}\label{LematCoupling}
Let $\mathcal{A}$ be an increasing property, $mp^{2} < 1$, and
\begin{equation}
\hat{p}_-=
\begin{cases}
mp^2\left(1-(n-2)p-\frac{mp^2}{2}\right)&\text{ for }np=o(1);\\
\frac{mp}{n}\left(1-\frac{\omega}{\sqrt{mnp}}-\frac{2}{np}-\frac{mp}{2n}\right)&\text{ for }np\to\infty\\ &\text{ and some } \omega\to\infty, \omega=o(\sqrt{mnp}). 
\end{cases}
\end{equation}
If 
$$
\Pra{\Gn{\hat{p}_-}\in\mathcal{A}}\to 1,
$$
then
\begin{equation}
\Pra{\Gnm{p}\in \mathcal{A}}\to 1.
\end{equation}
\end{lem}
\noindent Lemma~\ref{LematCoupling} combined with results presented in Section~\ref{SectionGnp} implies part (ii) and (ii') of the theorems. 
\end{proof}

\begin{rem}
In the case $\alpha<1$ it is simple to strengthen Lemma~\ref{LemmaStopienMaleAlfa}. Therefore having in mind the 'minimum degree phenomenon' we may conjecture, that the threshold function given in Theorem~\ref{TheoremAlphaLe1} may be tightened. The lemma and the conjecture are stated below.
\end{rem} 
\begin{lem}\label{LemmaStopienMaleAlfa2}
Let $\alpha < 1$ and 
$$
p_1=\frac{\ln n + \omega}{m}.
$$ 
If $\omega\to \infty$
then {\whp}
$\delta(\Gnm{p_1})\ge (1+o(1))n\ln n/m$.
\end{lem}
\begin{conj}
Let $\alpha < 1$
$$
p=\frac{\ln n + \omega}{m} 
$$
and $\omega\to \infty$. Then {\whp}
$\Gnm{p}$ is $k$-connected for any constant $k$ and contains a Hamilton cycle. 
\end{conj}
The stated conjecture contains assumption that $\alpha<1$. We believe that the case $\alpha=1$ is more complex. To support the thesis we give results concerning degree distribution~\cite{GpDegree} and phase transition~\cite{GpPhaseTransition2} for $\alpha=1$. Although they regard $p$ near phase transition threshold, they show that, for some properties there is a value of $\alpha$ for which the analysis of $\Gnm{p}$ is complicated.

\section{Proof of Lemma~\ref{LemmaMinDegree} and \ref{LemmaStopienMaleAlfa2}}\label{SectionMinimumDegree}

In the proofs we will use Chernoff bound (see Theorem 2.1 in \cite{KsiazkaJLR})
\begin{lem}
Let $X_n$ be a sequence of random variables with binomial distribution and $t_n>0$.\\
Then
\begin{equation}\label{Chernoff}
\textstyle \Pra{|X_n-\E X_n|\ge t_n}\le 2\exp\left(-\frac{3 t_n^2}{2 (3\E X_n + t_n)}\right).
\end{equation}
\end{lem}

\begin{proof}[Proof of Lemma~\ref{LemmaStopienMaleAlfa2}]
Let $w\in \W$.
Denote by $V_w$ the set of vertices which have chosen feature $w$.
Under the assumptions of the lemma by Lemma~\ref{LemmaStopienMaleAlfa} {\whp} for all $v\in \V$ there exists at least one $w\in\W$ such that $v\in V_w$. By definition of $\Gnm{p}$ if  $v\in V_w$, then $deg(v)\ge |V_w|$. Therefore the result follows by Chernoff bound \eqref{Chernoff} used for the random variable $Y_w=|V_w|$.
\end{proof}

\begin{proof}[Proof of Lemma~\ref{LemmaMinDegree}]
The part (ii) is easily obtained by the first moment method (see for example \cite{KsiazkaErdosSpencer,KsiazkaJLR}), since the expected number of the vertices of degree at most $k-1$ tends to zero for $\omega\to \infty$. Moreover we need only part (i) of the lemma, therefore we will concentrate on it. Although the proof is rather standard application of the second moment method (see~\cite{KsiazkaErdosSpencer,KsiazkaJLR}) we give it for completeness of considerations. 

We will assume that $\omega=o(\ln n)$. For other values of $\omega$ we may use the fact that having the minimum degree at least $k$ is an increasing property and a simple coupling argument (see Facts~\ref{FaktCouplingGnmp} and~\ref{FaktCouplingWlasnosci}).

The vertex degree analysis becomes complex for $\alpha$ near $1$ due to edge dependencies. Therefore, to simplify arguments, we will not study the degree directly but an auxiliary random variable, which approximates the degree of the vertex. Let $\mathcal{B}(n,m,p_k)$ be the random bipartite graph with bipartition $(\V,\W)$ in which $v$ and $w$ ($v\in\V$,$w\in\W$) are connected by an edge if and only if $w$ is a feature of $v$ in $\Gnm{p_k}$(i.e. $w\in W_v$). Note that by definition in $\mathcal{B}(n,m,p_k)$ each edge between $\V$ and $\W$ appears independently with probability $p_k$. Let $Z_v$, $v\in \V$, be a random variable counting edges between $W_v$ and $\V\setminus\{v\}$ in $\mathcal{B}(n,m,p_k)$. Let moreover $\xi_{v}$, $v\in\V$, be an indicator random variable such that
$$
\xi_{v}=
\begin{cases}
1,&\text{ if } Z_v=k-1;\\
0,&\text{ otherwise} 
\end{cases}
$$ 
and let
$$
\xi=\sum_{v\in\V} \xi_v.
$$

Surely, if $\xi_v=1$, then $deg(v)\le k-1$. Therefore we only need to prove that
$$
\Pra{\xi>0}\to 1.
$$ 
For that we will use the second moment method, i.e. we will show that
\begin{equation}\label{RownanieDrugiMoment}
\E \xi \to \infty
\quad
\text{and}
\quad
\E \xi (\xi - 1) \le (1+o(1))(\E \xi)^2.
\end{equation}
We will use the fact that for any $v,v'\in \V$
\begin{align*}
&\E \xi = n \Pr\{Z_v=k-1\}\\
&\E \xi (\xi - 1) = n(n-1) \Pr\{Z_v=k-1,Z_{v'}=k-1\}. 
\end{align*}
Therefore in order  to show \eqref{RownanieDrugiMoment} we will prove that
$$
n \Pr\{Z_v=k-1\}\to \infty\\ 
$$
and
\begin{multline*}
\Pra{Z_v=k-1,Z_{v'}=k-1}\le\\
 (1+o(1))\Pra{Z_v=k-1\}\Pr\{Z_{v'}=k-1}+o\left(\frac{1}{n^2}\right).
\end{multline*}

Given the value of a random variable $X_v=|W_v|$, i.e. given $X_v=x$, the random variable $Z_v$ has binomial distribution $\Bin{(n-1)x}{p_k}$. $X_v$ is also a binomial random variable, therefore if we set
$$
x_{\pm}=mp_k\left(1\pm \sqrt{\frac{5\ln n}{mp_k}}\right),
$$
then
by Chernoff bound~\eqref{Chernoff} 
$$
\Pra{x_-\le X_v \le x_+}=1-o\left(\frac{1}{n^2}\right).
$$
and $x_{\pm}=mp_k(1+o(1/\ln n))$.\\ 
Thus
\begin{align}
\nonumber n&\textstyle \Pra{Z_v=k-1}=\\ \label{RownanieWartoscOczekiwana}
&\quad=
\textstyle n\sum_{x=x_-}^{x_+}\Pra{Z_v=k-1|X_v=x}\Pra{X_v=x} + o\left(\frac{1}{n^2}\right) \ge\\ \nonumber
&\quad\ge\textstyle  n\binom{(n-1)x_-}{k-1}p_k^{k-1}(1-p_k)^{nx_+}\sum_{x=x_-}^{x_+}\Pra{X_v=x}=\\ \nonumber
&\quad=\textstyle n\frac{\left(nmp_k^2\left(1+o\left(\frac{1}{\ln n}\right)\right)\right)^{k-1}}{(k-1)!}
\exp\left(-nmp_k^2\left(1+o\left(\frac{1}{\ln n}\right)\right)\right)
\cdot\\ \nonumber
&\quad\textstyle 
\hspace{6.5cm}\cdot\Pra{x_-\le X_v \le x_+}=\\ \nonumber
&\quad=\textstyle \frac{1}{(k-1)!}\exp\left(\ln n + (k-1)\ln (nmp_k^2) - nmp_k^2 + o(1)\right)
\cdot\\ \nonumber
&\quad\textstyle 
\hspace{6.5cm}\cdot\Pra{x_-\le X_v \le x_+}=\\ \nonumber
&\quad=\textstyle \frac{1}{(k-1)!}\exp\left(-\omega+o(1)\right)(1+o(1))\to \infty.
\end{align}

\noindent Let $v,v'\in\V$ and
$
S=|W_{v}\cap W_{v'}|
$.
Notice that if $i\in \{0,1,2\}$ and $x,x'\in [x_-;x_++2]$, then uniformly over all $x,x'$
\begin{align*}
&\Pra{X_{v'}=x'+i,X_v=x+i,S=i}=\\
&\quad\textstyle=
\Pra{X_{v'}=x'+i}\Pra{X_v=x+i}\Pra{S=i|X_{v'}=x'+i,X_v=x+i}
=\\
&\quad\textstyle=
\binom{m}{x'+i}p_k^{x'+i}(1-p_k)^{m-x'-i}
\binom{m}{x+i}p_k^{x+i}(1-p_k)^{m-x-i}
\frac{\binom{x'+i}{i}\binom{m-x'-i}{x}}{\binom{m}{x+i}}
=\\
&\quad\textstyle=
(1+o(1))\left(\frac{mp_k}{x'}\right)^{i}\left(\frac{mp_k}{x}\right)^{i}
\Pra{X_{v'}=x'}\Pra{X_v=x}
\frac{1}{i!}\left(\frac{x\cdot x'}{m}\right)^i
=\\
&\quad\textstyle=
(1+o(1))\Pra{X_{v'}=x'}\Pra{X_v=x}\binom{m}{i}p_k^{2i}(1-p_k^{2})^{m-i}
=\\
&\quad\textstyle=
(1+o(1))\Pra{X_{v'}=x'}\Pra{X_v=x}\Pra{S=i}.
\end{align*}

\noindent Given $i\in \{0,1,2\}$ and $x,x'\in [x_-;x_++2]$ denote by $\mathcal{H}(x,x',i)$ event\linebreak $\{X_v=x+i,X_{v'}=x'+i,S=i\}$. Since for $1\le t\le k-1$, constant $k$ and $\alpha>1$
$$
\textstyle p_k^{-t}n^{t}\frac{1}{(nx)^t(nx')^t}=(1+o(1))\left(\frac{1}{nm^2p_k^3}\right)^t=(1+o(1))\frac{n^{t/2}}{m^{t/2}(\ln n)^{3t/2}}=o(1)
$$
and
$$
(1-p_k)^{i(n-1)}=(1+o(1)),
$$
 we have uniformly over all $x,x'\in [x_-;x_+ + 2]$
\begin{align*}
&\hspace{-0.5cm}\Pra{Z_v=k-1,Z_{v'}=k-1|\mathcal{H}(x,x',i)}=\\
&\quad\textstyle=\sum_{t=0}^{k-1}
\binom{i(n-1)}{t}p_k^{t}(1-p_k)^{i(n-1)-t}\cdot\\
&\quad\textstyle\hspace{1.5cm}\cdot\binom{(n-1)x}{k-1-t}p_k^{k-1-t}(1-p_k)^{(n-1)x - k + 1 + t}\cdot\\
&\quad\textstyle\hspace{1.5cm}\cdot 
\binom{(n-1)x'}{k-1-t}p_k^{k-1-t}(1-p_k)^{(n-1)x' - k + 1 + t}=\\ 
&\quad\textstyle=\binom{(n-1)x}{k-1}p_k^{k-1}(1-p_k)^{(n-1)x-k+1}\cdot\\
&\quad\textstyle\quad\cdot\binom{(n-1)x'}{k-1}p_k^{k-1}(1-p_k)^{(n-1)x'-k+1}\cdot\\
&\quad\textstyle\quad\cdot(1-p_k)^{i(n-1)}\cdot\\
&\quad\textstyle\quad\cdot\sum_{t=0}^{k-1}
(1-p_k)^{t}p_k^{-t}\binom{i(n-1)}{t}\frac{(k-1)_t}{((n-1)x-k+1+t)_t}\frac{(k-1)_t}{((n-1)x'-k+1+t)_t}=\\
&\quad\textstyle=(1+o(1))\Pra{Z_v=k-1|X_v=x}\Pra{Z_{v'}=k-1|X_{v'}=x}.
\end{align*}

Moreover $S$ has binomial distribution $\Bin{m}{p_k^2}$, therefore 
$$\textstyle \Pra{S\ge 3}=o(1/n^2).$$

\noindent Denote  $J=[x_-+2,x_+]$. Uniformly over all $x,x'\in J$
Therefore by Chernoff bound

\begin{multline*}
\Pra{X_v\notin J\text{ or } X_{v'}\notin J \text{ or } S\notin\{0,1,2\}}\le\\ \le
\Pra{X_v\notin J} + \Pra{X_v\notin J} + \Pra{S\ge 3} = o\left(\frac{1}{n^2}\right).
\end{multline*}

\noindent Finally by above calculation and \eqref{RownanieWartoscOczekiwana}

\begin{align*}
&\textstyle\hspace{-0.5cm}\Pra{Z_v=k-1,Z_{v'}=k-1}\le\\
&\textstyle\le\sum_{x=x_-}^{x_+}\sum_{x'=x_-}^{x_+}\sum_{i=0}^{2}
\Pra{Z_v=k-1,Z_{v'}=k-1|\mathcal{H}(x,x',i)}\cdot\\
&\textstyle\hspace{4.2cm}\cdot\Pra{X_{v'}=x'+i,X_v=x+i,S=i}+\\
&\textstyle\quad+\Pra{X_v\notin J\text{ or } X_{v'}\notin J \text{ or } S\notin\{0,1,2\}}\le\\
&\textstyle\le 
(1+o(1))
\left(\sum_{x=x_-}^{x_+}\Pra{Z_v=k-1|X_v=x}\Pra{X_v=x}\right)\cdot\\
&\textstyle\quad\cdot\left(\sum_{x'=x_-}^{x_+}\Pra{Z_{v'}=k-1|X_{v'}=x}\Pra{X_{v'}=x'}\right)\cdot\\
&\textstyle\quad\cdot\left(\sum_{i=0}^{2}\Pra{S_0=i}\right) + o\left(\frac{1}{n^2}\right)=\\
&\textstyle=(1+o(1))
\left(\Pra{Z_v=k-1}+o\left(\frac{1}{n^2}\right)\right)
\left(\Pra{Z_{v'}=k-1}+o\left(\frac{1}{n^2}\right)\right)
+\\
&\textstyle\quad+
o\left(\frac{1}{n^2}\right)=\\
&\textstyle=(1+o(1))
\Pra{Z_v=k-1}
\Pra{Z_{v'}=k-1}
+
o\left(\frac{1}{n^2}\right).  
\end{align*}

\end{proof}

\section{Proof of Lemma~\ref{LematCoupling}}\label{SectionCoupling}

We will begin the proof of Lemma~\ref{LematCoupling} by presenting auxiliary definitions and facts.

\subsection{An auxiliary graph $\G{M}$}
In the proof of the coupling's existence we will need an auxiliary graph. Let $M$ be a random variable with values in the set of positive integers (in the simplest case $M$ will be a given positive integer with probability one). By $\G{M}$ we will denote the random graph with the vertex set $\V$ and an edge set constructed by sampling $M$ times with repetition elements from the set of all two element subsets of $\V$. More precisely, in order to construct an edge set of $\G{M}$, first we choose the value of $M$ according to its probability distribution and then, given $M=t$, we sample $t$ times with repetition elements from the set of all two element subsets of $\V$. A subset $\{v,v'\}$ is an edge in $\G{M}$ if and only if it has been sampled at least once. For simplicity of notation if $M$ equals constant $t$ with probability one, has binomial or Poisson distribution we will write $\G{t}$, $\G{\Bin{\cdot}{\cdot}}$ or $\G{\Po{\cdot}}$, respectively.

\subsection{Coupling}

By the coupling $(G_1,G_2)$ of two random variables $G_1$ and $G_2$ we will mean a choice of the probability space on which we define a random vector $(G_1',G_2')$, such that $G_1'$ and $G_2'$ have the same distributions as $G_1$ and $G_2$, respectively. For simplicity of notation we will not differentiate between random variables $G_1',G_2'$ and $G_1,G_2$.

Let $G_1$ and $G_2$ be two random graphs. We will write
$$
G_1\coupling G_2\quad\text{ and }\quad G_1\coup G_2,
$$
if 
there exists a coupling $(G_1,G_2)$, such that under the coupling $G_1$ is a subgraph of $G_2$ with probability $1$ or $1-o(1)$, respectively. \\
Moreover we will write 
$$
G_1=G_2,
$$
if $G_1$ and $G_2$ have the same probability distribution (equivalently there exists a coupling $(G_1,G_2)$ such that $G_1=G_2$ with probability one.)

The facts stated below will be useful in the proof. 
A simple calculation shows (see \cite{GpEquivalence}) that in $\G{\Po{\lambda}}$ each edge appears independently with probability $1-\exp(-\lambda/{\textstyle \binom{n}{2}})$, therefore
\begin{fact}
\begin{equation}\label{RownanieGnGgwiazdka}
\G{\Po{\lambda}}=\Gn{1-\exp(-\lambda/{\textstyle \binom{n}{2}})}.
\end{equation}
\end{fact}
\noindent Since it is simple to construct suitable couplings we state the following facts without proof.
\begin{fact}\label{FaktSumaGgwiazdka}
Let $M_1\ldots M_m$ be independent random variables, then
a sum of $m$ independent graphs: 
$$\bigcup_{i=1}^m\G{M_i}=\G{\sum_{i=1}^{m}M_i}.$$
\end{fact}
\begin{fact}\label{FaktCouplingGnp}
If $\hat{p}\le \hat{p}'$, then
\begin{equation*}
\Gn{\hat{p}}\coupling  \Gn{\hat{p}'}.
\end{equation*}
\end{fact}
\begin{fact}\label{FaktCouplingGnmp}
If $p\le p'$, then
\begin{equation*}
\Gnm{p}\coupling  \Gnm{p'}.
\end{equation*}
\end{fact}
\begin{fact}\label{FaktCouplingGgwiazdka}
(i) 
Let $M_n$ be a sequence of random variables and let $a_n$ be a sequence of numbers. If 
\begin{equation*}
\Pra{M_n\ge a_n}=o(1)\quad (\Pra{M_n\le a_n}=o(1)), 
\end{equation*}
then
\begin{equation*}
\G{M_n}\coup \G{a_n}\quad (\G{a_n}\coup \G{M_n}). 
\end{equation*}
(ii) If a random variable $M$ is stochastically dominated by $M'$ (i.e. $M\prec M'$), then 
$$
\G{M}\coupling \G{M'}.
$$
\end{fact}
\begin{fact}\label{FaktCouplingIndependent}
Let $(G_i)_{i=1,\ldots,m}$ and $(G'_i)_{i=1,\ldots,m}$ be sequences of independent random graphs. If
$$
G_i\coupling G_i', \text{ for all }i=1,\ldots,m
$$
then
$$
\bigcup_{i=1}^{m}G_i\coupling \bigcup_{i=1}^{m}G'_i. 
$$  
\end{fact}
\begin{proof}
The proof is analogous to the proof of  Fact 2 in \cite{GpEquivalence2}.
\end{proof}

\begin{fact}\label{FaktCouplingWlasnosci}
Let $G_1$ and $G_2$ be two random graph models and $\mathcal{A}$ be an increasing property.
Let 
\begin{equation}\label{RownanieG1G2}
G_1\coupling G_2 \quad\text{ or }\quad G_1\coup G_2.
\end{equation}
If
$$
\Pra{G_1\in \mathcal{A}}\to 1,
$$ 
then
$$
\Pra{G_2\in \mathcal{A}}\to 1.
$$
\end{fact}
\begin{proof}
Under the coupling $(G_1,G_2)$ given by \eqref{RownanieG1G2} define event 
$$
\mathcal{H}:=\{G_1\subseteq G_2\}.
$$
Then 
\begin{align*}
1\ge \Pra{G_2\in \mathcal{A}}
&\ge\Pra{ G_2\in \mathcal{A}|\mathcal{H}\}\Pr\{\mathcal{H}}\ge
\\
&\ge\Pra{ G_1\in \mathcal{A}|\mathcal{H}\}\Pr\{\mathcal{H}}=
\\
&=\Pra{ \{G_1\in \mathcal{A}\}\cap\mathcal{H}}=\\
&=\Pra{ G_1\in \mathcal{A}}+\Pra{ \mathcal{H}}-\Pra{ \{G_1\in \mathcal{A}\}\cup\mathcal{H}}\ge\\
&\ge \Pra{ G_1\in \mathcal{A}}+\Pra{\mathcal{H}}-1=1+o(1).
\end{align*}
\end{proof}

\subsection{Total variation distance}

We will also use the notion of the total variation distance. Let $M_1$ and $M_2$ be two random variables with values in a countable set $A$, by the total variation distance we mean
\begin{multline*}
\dtv{M_1}{M_2}=\max_{A'\subseteq A} |\Pra{M_1\in A'}-\Pra{M_2\in A'}|=\\
\sum_{a\in A}|\Pra{M_1=a}-\Pra{M_2=a}|.
\end{multline*}
The following facts were shown in \cite{GpEquivalence} (see also \cite{GpEquivalence2}).
\begin{fact}\label{FaktDtv}
Let $M_1$ have binomial distribution $\Bin{m}{\hat{p}}$ and $M_2$ have Poisson distribution $\Po{m\hat{p}}$, then
\begin{multline*}
\dtv{\G{M_1}}{\Gn{1-\exp\left(-m\hat{p}/{\textstyle \binom{n}{2}}\right)}}=\\
=\dtv{\G{M_1}}{\G{M_2}}
\le 2\dtv{M_1}{M_2}\le 2\hat{p}.
\end{multline*}
\end{fact}
\begin{fact}\label{FaktDtvWlasnosci}
Let $G_1$ and $G_2$ be two random graphs, $a\in [0;1]$ and $\mathcal{A}$ be any graph property. If
$$
\dtv{G_1}{G_2}=o(1),
$$  
then
$$
\Pra{G_1\in\mathcal{A}}\to a\quad\text{ iff }\quad\Pra{G_2\in\mathcal{A}}\to a.
$$
\end{fact}

\subsection{Chernoff bound for Poisson distribution}

Let  $X_n$ and $X_n'$ be random variables with binomial $\Bin{\lambda_n n^{i+1}}{1/n^{i+1}})$ and Poisson $\Po{\lambda_n}$ distribution, respectively. Using fact that $\dtv{X_n}{X_n'}\le 1/n^{i+1}=o(1/n^{i})$ and \eqref{Chernoff} we get the following lemma. 
\begin{lem}
Let $X'_n$ be a sequence of random variables with Poisson distribution $\Po{\lambda_n}$ and $t_n>0$.\\
Then
\begin{equation}\label{ChernoffPoisson}
\textstyle \Pra{|X'_n-\lambda_n|\ge t_n}\le 2\exp\left(-\frac{3 t_n^2}{2 (3\lambda_n + t_n)}\right)+o\left(\frac{1}{n^{i}}\right).
\end{equation}
\end{lem}

\subsection{The coupling Lemma}

Now we will prove Lemma~\ref{LematCoupling}, which implies  part (ii) and (ii') of the theorems presented in Section~\ref{SectionResult}.

\begin{proof}[Proof of Lemma~\ref{LematCoupling}]
Let $w\in\W$. Denote by $V_w$ the set of vertices, which have chosen feature $w$ and $X_w=|V_w|$. Let  $\mathcal{G}[V_w]$ be a graph with the vertex set~$\V$ and an edge set constituted of these edges which have both ends in $V_w$. It is simple to construct a coupling $(\G{\lfloor X_w/2\rfloor},\mathcal{G}[V_w])$ which implies 
$$
\G{\lfloor X_w/2\rfloor}\coupling\mathcal{G}[V_w].
$$
Namely, given value of $X_w$, first we generate an instance $G$ of $\G{\lfloor X_w/2\rfloor}$. Let $Y_w$ be the number of non--isolated vertices in $G$. By definition $Y_w$ is at most $X_w$, therefore $V_w$ may be chosen to be a sum of the set of non--isolated vertices in $G$ and $X_w-Y_w$ vertices chosen uniformly at random from the remaining vertices. 

Therefore, since graphs $\G{\lfloor X_w/2\rfloor}$, $w\in\W$, are independent and $\mathcal{G}[V_w]$, $w\in\W$, are independent,  by Fact~\ref{FaktCouplingIndependent} and definition of $\Gnm{p}$  
\begin{equation}\label{RownanieGwGnmp}
\bigcup_{w\in\W}\G{\lfloor X_w/2\rfloor}\coupling \bigcup_{w\in\W}\mathcal{G}[V_w]=\Gnm{p}
\end{equation}
and by Fact~\ref{FaktSumaGgwiazdka}
\begin{equation}\label{RownanieGgwiazdkaGw}
\G{{\textstyle \sum_{w\in\W}\lfloor X_w/2\rfloor}}=\bigcup_{w\in\W}\G{\lfloor X_w/2\rfloor}.
\end{equation}

\noindent Now consider two cases

\noindent
{\bf CASE 1: $np=o(1)$.}\\
Notice that
\begin{equation*}
\sum_{w\in\W} \mathbb{I}_w\prec \sum_{w\in\W}\lfloor X_w/2\rfloor, 
\end{equation*}
where
$$
\mathbb{I}_w=
\begin{cases}
1,&\text{ if } X_w\ge 2;\\
0,&\text{otherwise}. 
\end{cases}
$$
The random variable $Z_1=\sum_{w\in\W} \mathbb{I}_w$ has binomial distribution $\Bin{m}{q}$, where $q=\Pra{X_w\ge 2}$, therefore by Fact~\ref{FaktCouplingGgwiazdka}(ii)
\begin{equation}\label{Rownanienpmale1}
\G{\Bin{m}{q}}\coupling\G{{\textstyle \sum_{w\in\W}\lfloor X_w/2\rfloor}}. 
\end{equation}
By Fact~\ref{FaktDtv}
\begin{equation}\label{Rownanienpmale2}
\dtv{\G{\Bin{m}{q}}}{\Gn{1-\exp(-mq/{\textstyle \binom{n}{2}})}}. 
\end{equation}
Moreover $q\ge \Pra{X_w=2} = \binom{n}{2}p^{2}(1-p)^{n-2}$ and $1-\exp(-x)\ge x - x^2/2$ for $x < 1$, thus $1-\exp(-mq/{\textstyle \binom{n}{2}})\ge mp^2\left(1-(n-2)p-\frac{mp^2}{2}\right)$. Therefore by Fact~\ref{FaktCouplingGnp}
\begin{equation}\label{Rownanienpmale3}
\Gn{p_-}\coupling \Gn{1-\exp(-mq/{\textstyle \binom{n}{2}})}. 
\end{equation}
Equations \eqref{RownanieGwGnmp}, \eqref{RownanieGgwiazdkaGw}, \eqref{Rownanienpmale1}, \eqref{Rownanienpmale2} and \eqref{Rownanienpmale3} combined with Facts~\ref{FaktCouplingWlasnosci} and~\ref{FaktDtvWlasnosci} imply the result. 

\noindent
{\bf CASE 2: $np\to \infty$.}\\
 Notice that
\begin{equation*}
\frac{Z_2}{2}-m
\prec \sum_{w\in\W}\lfloor X_w/2\rfloor, 
\end{equation*}
where $Z_2=\sum_{w\in\W}X_w$ has binomial distribution $\Bin{nm}{p}$. By Fact~\ref{FaktCouplingGgwiazdka}(ii)
\begin{equation}\label{Rownanienpduze1}
\G{\frac{Z_2}{2}-m}\coupling\G{{\textstyle \sum_{w\in\W}\lfloor X_w/2\rfloor}}. 
\end{equation}
By Chernoff bound \eqref{Chernoff} for any function $\omega\to \infty$, $\omega=o(\sqrt{nmp})$
\begin{multline*}
\Pra{\frac{Z_2}{2}-m\le \frac{nmp}{2}\left(1-\frac{\omega}{2\sqrt{nmp}}-\frac{2}{np}\right)}=\\
=\Pra{Z_2\le nmp-\frac{\omega\sqrt{mnp}}{2}}=o(1).
\end{multline*}
Moreover, by \eqref{ChernoffPoisson} for the random variable $Z_3$ with Poisson distribution \linebreak
 $\Po{\frac{nmp}{2}\left(1-\frac{\omega}{\sqrt{nmp}}-\frac{2}{np}\right)}$ we have
\begin{multline*}
\Pra{Z_3\ge \frac{nmp}{2}\left(1-\frac{\omega}{2\sqrt{nmp}}-\frac{2}{np}\right)}
=\\=
\Pra{Z_3\ge \mathbb{E} Z_3 + \frac{\omega\sqrt{nmp}}{4}}=o(1).
\end{multline*}
Therefore by Fact~\ref{FaktCouplingGgwiazdka}(i) used twice
\begin{equation}\label{Rownanienpduze2}
\G{\Po{{\textstyle \frac{nmp}{2}\left(1-\frac{\omega}{\sqrt{nmp}}-\frac{2}{np}\right)}}}\coup \G{\frac{Z_2}{2}-m}. 
\end{equation}
By \eqref{RownanieGnGgwiazdka} 
\begin{multline}\label{Rownanienpduze3}
\Gn{1-\exp\left(-{\textstyle \frac{mp}{n-1}\left(1-\frac{\omega}{\sqrt{nmp}}-\frac{2}{np}\right)}\right)}
=\\
=\G{\Po{{\textstyle \frac{nmp}{2}\left(1-\frac{\omega}{\sqrt{nmp}}-\frac{2}{np}\right)}}}.
\end{multline}
By Fact~\ref{FaktCouplingGnp} and 
$$
1-\exp\left(-{\textstyle \frac{mp}{n-1}\left(1-\frac{\omega}{\sqrt{nmp}}-\frac{2}{np}\right)}\right)
\ge 
\textstyle \frac{mp}{n}\left(1-\frac{\omega}{\sqrt{nmp}}-\frac{2}{np}-\frac{mp}{2n}\right)
$$
we have
\begin{equation}\label{Rownanienpduze4}
\Gn{p_-}\coupling \Gn{1-\exp\left(-{\textstyle \frac{mp}{n-1}\left(1-\frac{\omega}{\sqrt{nmp}}-\frac{2}{np}\right)}\right)}.
\end{equation}
Equations \eqref{RownanieGwGnmp}, \eqref{RownanieGgwiazdkaGw}, \eqref{Rownanienpduze1}, \eqref{Rownanienpduze2}, \eqref{Rownanienpduze3}  and \eqref{Rownanienpduze4} combined with Fact~\ref{FaktCouplingWlasnosci}  imply the result. 
\end{proof}

\section*{Acknowledgements}

I would like to thank colleagues attending our seminar for their helpful remarks, which allowed me to improve the layout of the paper and remove some ambiguities. 

\bibliographystyle{plain}
\normalfont
\bibliography{RIGCoupling}

\end{document}